\documentclass[11pt,twoside]{amsart}
\usepackage{latexsym,amssymb,amsmath}
\numberwithin{equation}{section}
\newtheorem{theorem}{Theorem}[section]
\newtheorem{lemma}[theorem]{Lemma}
\newtheorem{proposition}[theorem]{Proposition}

\theoremstyle{definition}
\newtheorem{definition}[theorem]{Definition} 
\newtheorem{problem}[theorem]{Problem} 
 
\newtheorem{remark}[theorem]{Remark}

\begin{document}


\title[Complete intersection
vanishing ideals]{Complete intersection
vanishing ideals on sets of clutter type over finite
fields}  

\author{Azucena Tochimani}
\address{
Departamento de
Matem\'aticas\\
Centro de Investigaci\'on y de Estudios
Avanzados del
IPN\\
Apartado Postal
14--740 \\
07000 Mexico City, D.F.
}
\email{tochimani@math.cinvestav.mx}
\thanks{The first author was partially supported by CONACyT. The
second author was partially supported by SNI}

\author{Rafael H. Villarreal}
\address{
Departamento de
Matem\'aticas\\
Centro de Investigaci\'on y de Estudios
Avanzados del
IPN\\
Apartado Postal
14--740 \\
07000 Mexico City, D.F.
}
\email{vila@math.cinvestav.mx}


\begin{abstract} 
In this paper we give a classification of complete intersection
vanishing ideals on parameterized sets of clutter type over finite
fields. 
\end{abstract}

\maketitle

\section{Introduction}\label{section-intro}
Let $R=K[\mathbf{y}]=K[y_1,\ldots,y_n]$ be a polynomial ring over 
a finite field $K=\mathbb{F}_q$ and 
let $y^{v_1},\ldots,y^{v_s}$ be a finite set of monomials in
$K[\mathbf{y}]$. As usual we denote the affine and projective spaces
over the field $K$ of dimensions $s$ and $s-1$ by
$\mathbb{A}^s$ and $\mathbb{P}^{s-1}$, respectively. Points of the
projective space ${\mathbb P}^{s-1}$ are denoted by $[\alpha]$, where $0\neq \alpha\in
\mathbb{A}^s$. 

We consider a set $\mathbb{X}$, in the projective space
$\mathbb{P}^{s-1}$, parameterized  by $y^{v_1},\ldots,y^{v_s}$. The set
$\mathbb{X}$ consists of all points $[(x^{v_1},\ldots,x^{v_s})]$ in $\mathbb{P}^{s-1}$ that 
are well defined, i.e., $x\in K^n$ and $x^{v_i}\neq 0$  for some $i$.
The set $\mathbb{X}$ is called of {\it clutter type} if ${\rm
supp}(y^{v_i})\not\subset{\rm supp}(y^{v_j})$ for $i\neq j$, where 
${\rm supp}(y^{v_i})$ is the {\it support\/} of the monomial $y^{v_i}$
consisting of the variables that occur in $y^{v_i}$. In this case we
say that the set of monomials $y^{v_1},\ldots,y^{v_s}$ is of {\it
clutter type\/}.  This terminology
comes from the fact that the condition ${\rm
supp}(y^{v_i})\not\subset{\rm supp}(y^{v_j})$ 
for
$i\neq j$ means that there is a {\it clutter\/} $\mathcal{C}$, in the
sense of \cite{ci-codes}, with
vertex set $V(\mathcal{C})=\{y_1,\ldots,y_n\}$ and edge set 
$$E(\mathcal{C})=\{{\rm supp}(y^{v_1}),\ldots, {\rm supp}(y^{v_s})\}.$$
A clutter is also called a {\it simple hypergraph\/}, see Definition~\ref{clutter-def}.
  
Let $S=K[t_1,\ldots,t_s]=\oplus_{d=0}^\infty S_d$ be a polynomial ring 
over the field $K$ with the standard grading. The graded ideal 
$I(\mathbb{X})$ generated by the  
homogeneous polynomials of $S$ that vanish at all points of
$\mathbb{X}$ is called the {\it vanishing ideal\/} of $\mathbb{X}$. 

There are good reasons to study vanishing ideals over finite fields. They are used in 
algebraic coding theory \cite{GRT} and in polynomial interpolation
problems \cite{gasca-sauer,vanishing-ideals}. The Reed-Muller-type codes
arising from vanishing ideals on monomial parameterizations have
received a lot of attention
\cite{carvalho,delsarte-goethals-macwilliams,geil-thomsen,GRT,cartesian-codes,
algcodes,ci-codes,sorensen}. 

The vanishing ideal $I(\mathbb{X})$ is a 
{\it complete intersection\/} if $I(\mathbb{X})$ is generated by 
$s-1$ homogeneous polynomials. Notice that $s-1$ is the height of
$I(\mathbb{X})$ in the sense of \cite{Mats}. The interest in complete
intersection vanishing 
ideals over finite
fields comes from information and communication theory, and algebraic
coding theory \cite{duursma-renteria-tapia,gold-little-schenck,hansen}. 

Let $T$ be a projective torus in $\mathbb{P}^{s-1}$ (see
Definition~\ref{projectivetorus-def}) and let $\mathbb{X}$ be the set
in $\mathbb{P}^{s-1}$ parameterized by a clutter $\mathcal{C}$ (see
Definition~\ref{parameterized-clutter}). Consider the set 
$X=\mathbb{X}\cap T$. In \cite{ci-codes} it is shown that
$I(X)$ is a complete intersection if and only if $X$ is a projective
torus in $\mathbb{P}^{s-1}$ . If the clutter $\mathcal{C}$ has all its 
edges of the same cardinality, in \cite{d-compl} a classification of the complete  
intersection property of $I(X)$ is given using linear algebra. 

The main result of this paper is a classification of the complete
intersection property of $I(\mathbb{X})$ when $\mathbb{X}$ is of
clutter type (Theorem~\ref{azucena-vila-ci}). Using the techniques
of \cite{algcodes}, this classification can be used to study the {\it basic
parameters\/} \cite{MacWilliams-Sloane,tsfasman} 
of the Reed-Muller-type codes associated to $\mathbb{X}$.

For all unexplained
terminology and additional information,  we refer to
\cite{Mats} (for commutative algebra), \cite{CLO} (for Gr\"obner bases), and
\cite{algcodes,vanishing-ideals,tsfasman} (for vanishing
ideals and coding theory). 

\section{Complete intersections}
In this section we give a full classification of the 
complete intersection property of vanishing ideals of sets of clutter
type over finite fields. We 
continue to employ the notations and 
definitions used in Section~\ref{section-intro}.

Throughout this section $K=\mathbb{F}_q$ is a finite field, 
$y^{v_1},\ldots,y^{v_s}$ are distinct monomials in
the polynomial ring $R=K[\mathbf{y}]=K[y_1,\ldots,y_n]$, with
$v_i=(v_{i1},\ldots,v_{in})$ and $y^{v_i}=y_1^{v_{i1}}\cdots
y_n^{v_{in}}$ for $i=1,\ldots,s$, $\mathbb{X}$ is the
set in $\mathbb{P}^{s-1}$ parameterized by these monomials, and $I(\mathbb{X})$ is the
vanishing ideal of $\mathbb{X}$. Recall that
$I(\mathbb{X})$ is the graded ideal of the polynomial ring $S=K[t_1,\ldots,t_s]$
generated by the homogeneous polynomials of $S$ that vanish on
$\mathbb{X}$.

\begin{definition} Given $a=(a_1,\ldots,a_n)\in\mathbb{N}^n$, we set 
$y^a:=y_1^{a_1}\cdots y_n^{a_n}$. The {\it support\/} of $y^a$, 
denoted ${\rm supp}(y^a)$, is the set of all $y_i$ such that $a_i>0$.
\end{definition}

\begin{definition} The set $\mathbb{X}$ is of {\it clutter type\/} if
${\rm
supp}(y^{v_i})\not\subset{\rm supp}(y^{v_j})$ for $i\neq j$.
\end{definition}

\begin{definition}\rm 
A {\it binomial\/} of $S$ is an
element of the form $f=t^a-t^b$, for some 
$a,b$ in $\mathbb{N}^s$. An ideal generated by 
binomials is called a {\it binomial ideal\/}.
\end{definition}

The set $\mathcal{S}=\mathbb{P}^{s-1}\cup\{[0]\}$ is a monoid under componentwise
multiplication, that is, given $[\alpha]=[(\alpha_1,\ldots,\alpha_s)]$
and $[\beta]=[(\beta_1,\ldots,\beta_s)]$ in $\mathcal{S}$, the operation of this monoid
is given by 
$$
[\alpha]\cdot[\beta]=[\alpha_1\beta_1,\cdots,\alpha_s\beta_s],
$$
where $[\mathbf{1}]=[(1,\ldots,1)]$ is the identity element.

\begin{theorem}{\rm\cite{vanishing-binomial}}\label{oct11-14-1} If
$K=\mathbb{F}_q$ is a finite
field and $\mathbb{Y}$ is a subset of $\mathbb{P}^{s-1}$, 
then  $I(\mathbb{Y})$ is a binomial ideal if and
only if $\mathbb{Y}\cup\{[0]\}$ is a submonoid of
$\mathbb{P}^{s-1}\cup\{[0]\}$.
\end{theorem}

\begin{remark} Since $\mathbb{X}$ is parameterized by monomials, the
set $\mathbb{X}\cup\{[0]\}$ is a monoid under componentwise
multiplication. Hence, by Theorem~\ref{oct11-14-1}, $I(\mathbb{X})$
is a binomial ideal.
\end{remark}

\begin{lemma}\label{nov10-14} Let $y^{v_1},\ldots,y^{v_s}$ be a set of monomials such
that ${\rm supp}(y^{v_i})\not\subset{\rm supp}(y^{v_j})$
for any $i\neq j$ and let $\mathcal{G}$ be a minimal generating set of
$I(\mathbb{X})$ consisting of binomials. The following hold.
\begin{itemize}
\item[(a)] If $0\neq f=t_j^{a_j}-t^c$ for some $1\leq j\leq s$
and some positive integer $a_j$, then $f\notin I(\mathbb{X})$.
\item[(b)] For each pair $1\leq i<j\leq s$, there is $g_{ij}$ in $\mathcal{G}$
such that $g_{ij}=\pm(t_i^{c_{ij}}t_j-t^{b_{ij}})$, where $c_{ij}$ is
a positive integer less than or equal to $q$ and $b_{ij}\in\mathbb{N}^s\setminus\{0\}$. 
\item[(c)] If $I(\mathbb{X})$ is a complete intersection, then 
$s\leq 4$.
\end{itemize}
\end{lemma}

\begin{proof} (a): We proceed by contradiction. Assume that $f$ is in
$I(\mathbb{X})$. Since $I(\mathbb{X})$ is a graded binomial ideal, the
binomial $f$ is homogeneous of degree $a_j$, otherwise $t_j^{a_j}$ and
$t^c$ would be in $I(\mathbb{X})$ which is impossible. Thus
$c\in\mathbb{N}^s\setminus\{0\}$. Hence, as $f\neq 0$, we can pick
$t_i\in{\rm supp}(t^c)$ with $i\neq j$. By hypothesis there is
$y_k\in{\rm supp}(y^{v_i})\setminus{\rm supp}(y^{v_j})$, i.e.,
$v_{ik}>0$ and $v_{jk}=0$. Making $y_k=0$ and $y_\ell=1$ for 
$\ell\neq k$, we get that $f(y^{v_1},\ldots,y^{v_s})=1$, a
contradiction.

(b): The binomial $h=t_i^qt_j-t_it_j^q$ vanishes at all points of
$\mathbb{P}^{s-1}$, i.e., $h$ is in $I(\mathbb{X})$. Thus there is $g_{ij}$
in $\mathcal{G}$ such that $t_i^qt_j$ is a multiple of one of the two terms
of the binomial $g_{ij}$. Hence, by part (a), the assertion follows. 

(c): Since $I(\mathbb{X})$ is a complete intersection, there is a set
of binomials $\mathcal{G}=\{g_1,\ldots,g_{s-1}\}$ that generate $I(\mathbb{X})$. The
number of monomials that occur in $g_1,\ldots,g_{s-1}$ is at most
$2(s-1)$. Thanks to part (b) for each pair $1\leq i<j\leq s$, there
is a monomial $t_i^{c_{ij}}t_j$, with $c_{ij}\in\mathbb{N}_+$, and a
binomial $g_{ij}$ in $\mathcal{G}$ such that the monomial
$t_i^{c_{ij}}t_j$ occurs in $g_{ij}$. As there are $s(s-1)/2$ of these
monomials, we get $s(s-1)/2\leq 2(s-1)$. Thus $s\leq 4$.
\end{proof}

\begin{lemma}\label{nov12-14} Let $K$ be a field and let $I$ be the ideal of $S=K[t_1,t_2,t_3,t_4]$ generated by
the binomials $g_1=t_1t_2-t_3t_4,g_2=t_1t_3-t_2t_4,g_3=t_2t_3-t_1t_4$.  
The following hold.
\begin{itemize}
\item[(i)] $\mathcal{G}=\{t_2t_3-t_1t_4, t_1t_3-t_2t_4, t_1t_2-t_3t_4, t_2^2t_4-t_3^2t_4,
      t_1^2t_4-t_3^2t_4, t_3^3t_4-t_3t_4^3\}$ is a Gr\"obner basis of
      $I$ with respect to the GRevLex order $\prec$ on $S$.  
\item[(ii)] If ${\rm char}(K)=2$, then ${\rm rad}(I)\neq I$.
\item[(iii)] If ${\rm char}(K)\neq 2$ and $e_i$ is the $i$-th unit
vector, then
$I=I(\mathbb{X})$, where 
$$
\mathbb{X}=\{[e_1],\,
[e_2],\,[e_3],\,[e_4],\,[(1,-1,-1,1)],\,[(1,1,1,1)],\,[(-1,-1,1,1)],\,[(-1,1,-1,1)]\}.
$$
\end{itemize}
\end{lemma}

\begin{proof} (i): Using Buchberger's criterion
\cite[p.~84]{CLO}, it is seen that $\mathcal{G}$ is a Gr\"obner basis
of $I$. 

(ii): Setting $h=t_1t_2-t_1t_3$, we get
$h^2=(t_1t_2)^2-(t_1t_3)^2=t_1t_2g_1+t_1t_3g_2$, where
$g_1=t_1t_2-t_3t_4$ and $g_2=t_1t_3-t_2t_4$. Thus $h\in{\rm rad}(I)$. Using part (i) it is seen
that $h\notin I$.

(iii): As $g_i$ vanishes at all points of $\mathbb{X}$ for $i=1,2,3$, we get the
inclusion $I\subset I(\mathbb{X})$. Since $\mathbb{X}\cup\{0\}$ is a
monoid under componentwise multiplication, by
Theorem~\ref{oct11-14-1}, $I(\mathbb{X})$ is a binomial ideal.  Take
a homogeneous binomial $f$ in $S$ that vanishes at all points of 
$\mathbb{X}$. Let  $h=t^a-t^b$, $a=(a_i)$, $b=(b_i)$, be the residue
obtained by dividing $f$ by $\mathcal{G}$. Hence we can write $f=g+h$,
where $g\in I$ and the terms $t^a$ and $t^b$ are not divisible by any of
the leading terms of $\mathcal{G}$. It suffices to show that $h=0$.
Assume that $h\neq 0$. As $h\in I(\mathbb{X})$ and $[e_i]$ is in $\mathbb{X}$ for
all $i$, we get that $|{\rm supp}(t^a)|\geq 2$ and $|{\rm
supp}(t^b)|\geq 2$. It follows that $h$ has one of the following
forms:
$$
\begin{array}{lll}
h=t_1t_4^i-t_2t_4^i,&h=t_1t_4^i-t_3t_4^i,&h=t_2t_4^i-t_3t_4^i,\\
h=t_3^2t_4^{i-1}-t_3t_4^i,&h=t_3^2t_4^{i-1}-t_2t_4^i,&h=t_3^2t_4^{i-1}-t_1t_4^i,
\end{array}
$$
where $i\geq 1$, a contradiction because none of
these binomials vanishes at all points of $\mathbb{X}$.
\end{proof}

\begin{definition}\label{clutter-def}\rm A {\it hypergraph\/} ${\mathcal H}$ is a pair
$(V(\mathcal{H}),E(\mathcal{H}))$ such that $V(\mathcal{H})$ is a
finite set and $E(\mathcal{H})$ is a subset of the 
set of all subsets of $V(\mathcal{H})$. The elements of
$E(\mathcal{H})$ and $V(\mathcal{H})$ are called {\it edges\/} and
{\it vertices\/}, respectively. 
A hypergraph is {\it simple\/} if $f_1\not\subset f_2$ for
any two edges $f_1,f_2$. A simple hypergraph is called a {\it
clutter\/} and will be denoted by
$\mathcal{C}$ instead of $\mathcal{H}$.  
\end{definition}

One example of a clutter is a graph with the vertices and edges defined in the 
usual way. 

\begin{definition}\label{parameterized-clutter} Let $\mathcal{C}$ be a {\it clutter\/} with vertex set 
$V(\mathcal{C})=\{y_1,\ldots,y_n\}$, let $f_1,\ldots,f_s$ be the edges of $\mathcal{C}$
and let $v_k=\sum_{x_i\in f_k}e_i$ be the {\it characteristic vector\/} of $f_k$
for $1\leq k\leq s$, where $e_i$ is the $i$-th unit vector. The set in the projective space $\mathbb{P}^{s-1}$ parameterized by
$y^{v_1},\ldots,y^{v_s}$, denoted by $\mathbb{X}_\mathcal{C}$,  
is called the {\it projective set parameterized} by $\mathcal{C}$. 
\end{definition}

\begin{lemma}\label{no-quartics}
Let $K=\mathbb{F}_q$ be a finite field with $q\neq 2$ elements, let
$\mathcal{C}$ be a clutter with vertices $y_1,\ldots,y_n$, let 
$v_1,\ldots,v_s$ be the characteristic vectors of the edges of
$\mathcal{C}$ and let $\mathbb{X}_\mathcal{C}$ be the projective set
parameterized by $\mathcal{C}$. If $f=t_it_j-t_kt_\ell\in
I(\mathbb{X}_\mathcal{C})$, with
$i,j,k,l$ distinct, then $y^{v_i}y^{v_j}=y^{v_k}y^{v_\ell}$. 
\end{lemma}

\begin{proof} For simplicity assume that
$f=t_1t_2-t_3t_4$.  Setting $A_1={\rm supp}(y^{v_1}y^{v_2})$,
$A_2={\rm supp}(y^{v_3}y^{v_4})$, $S_1={\rm supp}(y^{v_1})\cap {\rm
supp}(y^{v_2})$ and 
$S_2={\rm supp}(y^{v_3})\cap {\rm supp}(y^{v_4})$, it suffices to show
the equalities $A_1=A_2$ and $S_1=S_2$. If $A_1\not\subset A_2$, pick
$y_k\in A_1\setminus A_2$. Making $y_k=0$ and $y_\ell=1$ for
$\ell\neq k$, and using that $f$ vanishes on $\mathbb{X}_\mathcal{C}$, we get 
that $f(y^{v_1},\ldots,y^{v_4})=-1=0$, 
a contradiction. Thus $A_1\subset A_2$. The other inclusion follows by a similar
reasoning. Next we show the equality $S_1=S_2$. If $S_1\not\subset
S_2$, pick a variable $y_k\in S_1\setminus S_2$. Let $\beta$ be a
generator of the cyclic group $\mathbb{F}_q^*=\mathbb{F}_q\setminus\{0\}$. Making $y_k=\beta$, 
$y_\ell=1$ for $\ell\neq k$, and using that $f$ vanishes on
$\mathbb{X}_\mathcal{C}$ and the equality $A_1=A_2$, we get that
$f(y^{v_1},\ldots,y^{v_4})=\beta^2-\beta=0$. 
Hence  $\beta^2=\beta$ and $\beta=1$, a contradiction because $q\neq 2$. Thus $S_1\subset
S_2$. The other inclusion follows by a similar argument. 
\end{proof}

\begin{remark}\label{az12-macaulay2} 
Let $K=\mathbb{F}_q$ be a finite field with $q$ odd and let 
$\mathbb{X}$ be the set of clutter type in $\mathbb{P}^3$ parameterized by
the following monomials:
\begin{eqnarray*}
y^{v_1}&=&y_1^{q-1}y_2^ry_3^ry_4^{q-1}y_5^{q-1}y_6^{q-1}y_7^{q-1},\\ 
y^{v_2}&=&y_1^ry_2^ry_3^{q-1}y_4^{q-1}y_5^{q-1}y_6^{q-1}y_8^{q-1},\\
y^{v_3}&=&y_2^{q-1}y_4^{q-1}y_1^ry_3^ry_5^{q-1}y_7^{q-1}y_8^{q-1},\\
y^{v_4}&=&y_1^{q-1}y_2^{q-1}y_3^{q-1}y_4^{q-1}y_6^{q-1}y_7^{q-1}y_8^{q-1},
\end{eqnarray*}
where $r=(q-1)/2$. Then 
$$
\mathbb{X}=\{[e_1],\,
[e_2],\,[e_3],\,[e_4],\,[(1,-1,-1,1)],\,[(1,1,1,1)],\,[(-1,-1,1,1)],\,[(-1,1,-1,1)]\},
$$
$|\mathbb{X}|=8$ and 
$I(\mathbb{X})=(t_1t_2-t_3t_4,t_1t_3-t_2t_4,t_2t_3-t_1t_4)$. 
\end{remark}

Below we show that the set $\mathbb{X}$ of
Remark~\ref{az12-macaulay2} cannot be parameterized by a clutter.  

\begin{remark}\label{not-parameterized-by-clutter} 
Let $K=\mathbb{F}_q$ be a field with $q\neq 2$ elements. Then 
the ideal 
$$I=(t_1t_2-t_3t_4,t_1t_3-t_2t_4,t_2t_3-t_1t_4)$$
cannot be the vanishing ideal of a set in $\mathbb{P}^3$
parameterized by a clutter. Indeed assume that there is a clutter
$\mathcal{C}$ such that $I=I(\mathbb{X}_\mathcal{C})$ and
$\mathbb{X}_\mathcal{C}\subset\mathbb{P}^3$. If $v_1,\ldots,v_4$ are
the characteristic 
vectors of the edges of $\mathcal{C}$. Then,
by Lemma~\ref{no-quartics}, we get 
$v_1+v_2=v_3+v_4$, $v_1+v_3=v_2+v_4$ and $v_2+v_3=v_1+v_4$. It follows
that $v_1=v_2=v_3=v_4$, a contradiction.
\end{remark}

\begin{lemma}\label{nov12-14-1} Let $K$ be a field and 
let $I$ be the ideal of $S=K[t_1,t_2,t_3]$ generated by
the binomials $g_1=t_1t_2-t_2t_3,g_2=t_1t_3-t_2t_3$.  
The following hold.
\begin{itemize}
\item[(i)] $\mathcal{G}=\{t_1t_3-t_2t_3,\, t_1t_2-t_2t_3,\, t_2^2t_3-t_2t_3^2\}$ is a Gr\"obner basis of
      $I$ with respect to the GRevLex order $\prec$ on $S$.  
\item[(ii)] $I=I(\mathbb{X})$, where 
$\mathbb{X}=\{[e_1],\,
[e_2],\,[e_3],\, [(1,1,1)]\}$.
\end{itemize}
\end{lemma}

\begin{proof} It follows using the arguments given in
Lemma~\ref{nov12-14}.
\end{proof}

\begin{remark}\label{az13-mac2} 
Let $K=\mathbb{F}_q$ be a finite field with $q$ elements and let 
$\mathbb{X}$ be the projective set in $\mathbb{P}^2$ parameterized by
the following monomials:
$$
y^{v_1}=y_1^{q-1}y_2^{q-1},\, y^{v_2}=y_2^{q-1}y_3^{q-1},\, 
y^{v_3}=y_1^{q-1}y_3^{q-1}.
$$
Then $\mathbb{X}=\{[e_1],\,
[e_2],\,[e_3],\,[(1,1,1)]\}$ and 
$I(\mathbb{X})=(t_1t_2-t_2t_3,\, t_1t_3-t_2t_3)$. 
\end{remark}

\begin{definition}\label{projectivetorus-def} The set 
$T=\{[(x_1,\ldots,x_s)]\in\mathbb{P}^{s-1}\vert\, x_i\in
K^*\mbox{ for all }i\}$ is called a {\it projective torus} 
in $\mathbb{P}^{s-1}$. 
\end{definition}

\begin{lemma}\label{nov14-14} Let $\beta$ be a generator of
$\mathbb{F}_q^*$ and $0\neq r\in\mathbb{N}$. Suppose $s=2$.  
If $I=(t_1^{r+1}t_2-t_1t_2^{r+1})$ and  $r$ divides $q-1$,
then $I=I(\mathbb{X})$, where $\mathbb{X}$ is the set of clutter type
in $\mathbb{P}^1$ parameterized 
by $y_1^{q-1}$, $y_2^{q-1}y_3^k$ and $r={\rm o}(\beta^k)$. 
\end{lemma}

\begin{proof} We set $f=t_1^{r+1}t_2-t_1t_2^{r+1}$. Take a point
$P=[(x_1^{q-1},x_2^{q-1}x_3^k)]$ in $\mathbb{X}$. Then
$$
f(P)=(x_1^{q-1})^{r+1}(x_2^{q-1}x_3^k)-(x_1^{q-1})(x_2^{q-1}x_3^k)^{r+1}.
$$
We may assume $x_1\neq 0$ and $x_2\neq 0$. Then
$f(P)=x_3^k-(x_3^k)^{r+1}$. If $x_3\neq 0$, then $x_3=\beta^i$ for
some $i$ and $(x_3^k)^{r+1}=x_3^k$, that is, $f(P)=0$. Therefore one
has the inclusion $(f)\subset I(\mathbb{X})$.

Next we show the inclusion $I(\mathbb{X})\subset (f)$. By 
Theorem~\ref{oct11-14-1}, $I(\mathbb{X})$ is a binomial ideal. Take
a non-zero binomial $g=t_1^{a_1}t_2^{a_2}-t_1^{b_1}t_2^{b_2}$ that
vanishes on $\mathbb{X}$. Then $a_1+a_2=b_1+b_2$ because $I(\mathbb{X})$
is graded. We may assume that $b_1>a_1$ and $a_2>b_2$. We may also
assume that $a_1>0$ and $b_2>0$ because
$\{[e_1],[e_2]\}\subset\mathbb{X}$. Then
$g=t_1^{a_1}t_2^{b_2}(t_2^{a_2-b_2}-t_1^{b_1-a_1})$. As $g$ vanishes
on $\mathbb{X}$, making  $y_3=\beta$ and $y_1=y_2=1$, we get
$(\beta^k)^{a_2-b_2}=1$. 
Hence $a_2-b_2=\lambda r$ for some
$\lambda\in\mathbb{N}_+$, where $r={o}(\beta^k)$. Thus
$t_2^{a_2-b_2}-t_1^{b_1-a_1}$ is equal to $t_2^{\lambda r}-t_1^{\lambda
r}\in (t_1^r-t_2^r)$. Therefore $g$ is a multiple of
$f=t_1t_2(t_1^r-t_2^r)$ because $a_1>0$ and $b_2>0$. Thus $g\in(f)$.  
\end{proof}

\begin{lemma}\label{nov14-14-1} Let $K=\mathbb{F}_q$ be a finite field.
If $\{[e_1],[e_2]\}\subset\mathbb{Y}\subset\mathbb{P}^1$ and
$\mathbb{Y}\cup\{0\}$ is a monoid under componentwise multiplication,
then there is $0\neq r\in\mathbb{N}$ such that
$I(\mathbb{Y})=(t_1^{r+1}t_2-t_1t_2^{r+1})$ and $r$ divides $q-1$.
\end{lemma}

\begin{proof} We set $f=t_1^{r+1}t_2-t_1t_2^{r+1}$ 
and $X=\mathbb{Y}\cap T$, where $T$ is a projective torus in $\mathbb{P}^1$.
The set $X$ is a group, under componentwise multiplication, because $X$
is a finite monoid and the cancellation laws hold. By 
Theorem~\ref{oct11-14-1}, $I(\mathbb{Y})$ is a binomial ideal.
Clearly $(f)\subset I(\mathbb{Y})$. To show the other inclusion take
a non-zero binomial $g=t_1^{a_1}t_2^{a_2}-t_1^{b_1}t_2^{b_2}$ that
vanish on $\mathbb{Y}$. Then $a_1+a_2=b_1+b_2$ because $I(\mathbb{Y})$
is graded. We may assume that $b_1>a_1$ and $a_2>b_2$. We may also
assume that $a_1>0$ and $b_2>0$ because
$\{[e_1],[e_2]\}\subset\mathbb{X}$. Then
$g=t_1^{a_1}t_2^{b_2}(t_2^{a_2-b_2}-t_1^{b_1-a_1})$. The subgroup of
$\mathbb{F}_q^*$ given by $H=\{\xi\in\mathbb{F}_q^*\, \vert\,
[(1,\xi)]\in X\}$ has order $r=|X|$. Pick a generator $\beta$ of the cyclic
group $\mathbb{F}_q^*$. Then $H$ is a cyclic group generated by
$\beta^k$ for some $k\geq 0$. As $g$ vanishes on $\mathbb{Y}$, one has
that $t_2^{a_2-b_2}-t_1^{b_1-a_1}$ vanishes on $X$. In particular
$(\beta^k)^{a_2-b_2}=1$. Hence $a_2-b_2=\lambda r$ for some
$\lambda\in\mathbb{N}_+$, where $r={o}(\beta^k)=|X|$. 
Proceeding as in the proof of
Lemma~\ref{nov14-14} one derives that $g\in(f)$. Noticing that $T$ has
order $q-1$, we obtain that $r$ divides $q-1$.
\end{proof}

\begin{definition}\label{ci-def} An ideal $I\subset S$ is called a {\it complete intersection\/} if 
there exists $g_1,\ldots,g_{r}$ in $S $ such that $I=(g_1,\ldots,g_{r})$, 
where $r$ is the height of $I$. 
\end{definition}

Recall that a graded ideal $I$ is a complete
intersection if and only if $I$ is generated by a homogeneous regular 
sequence with ${\rm ht}(I)$ elements (see \cite[Proposition~2.3.19,
Lemma~2.3.20]{monalg-rev}). 

\begin{theorem}\label{azucena-vila-ci} Let $K=\mathbb{F}_q$ be a
finite field and let $\mathbb{X}$ be a set in
$\mathbb{P}^{s-1}$ parameterized by a
set of monomials $y^{v_1},\ldots,y^{v_s}$ such that ${\rm supp}(y^{v_i})\not\subset{\rm supp}(y^{v_j})$
for any $i\neq j$. Then $I(\mathbb{X})$ is a complete intersection if
and only if $s\leq 4$ and, up to permutation of variables,
$I(\mathbb{X})$ has one of the following forms:
\begin{itemize}
\item[(i)] $s=4$, $q$ is odd and 
$I=(t_1t_2-t_3t_4,t_1t_3-t_2t_4,t_2t_3-t_1t_4)$.
\item[(ii)] $s=3$ and $I=(t_1t_2-t_2t_3,t_1t_3-t_2t_3)$.
\item[(iii)] $s=2$ and $I=(t_1^{r+1}t_2-t_1t_2^{r+1})$, where
$0\neq r\in\mathbb{N}$ is a divisor of $q-1$.  
\item[(iv)] $s=1$ and $I=(0)$.
\end{itemize}
\end{theorem}

\begin{proof} $\Rightarrow$): Assume that $I(\mathbb{X})$ is a
complete intersection. By Lemma~\ref{nov10-14}(c) one has $s\leq 4$.

Case (i): Assume that $s=4$. Setting $I=I(\mathbb{X})$, by hypothesis $I$ is
generated by $3$ binomials $g_1,g_2,g_3$. By Lemma~\ref{nov10-14}(b) for each pair
$1\leq i<j\leq 4$ there are positive integers $c_{ij}$ and $a_{ij}$
such that $t_i^{c_{ij}}t_j$ and $t_it_j^{a_{ij}}$ occur as terms in
$g_1,g_2,g_3$. Since there are at most $6$ monomials that occur in 
the $g_i$'s, we get that $c_{ij}=a_{ij}=1$ for $1\leq i<j\leq 4$. 
Thus, up to permutation of variables, there are $4$ subcases to
consider: 
\begin{eqnarray*}
(a):\ \  g_1=t_1(t_2-t_3),&g_2=t_1t_4-t_2t_3,&g_3=t_4(t_2-t_3).\\
(b):\ \  g_1=t_1(t_2-t_3),&g_2=t_4(t_1-t_3),&g_3=t_2(t_3-t_4).\\ 
(c):\ \  g_1=t_1t_2-t_3t_4,&g_2=t_1t_3-t_2t_4,&g_3=t_2t_3-t_1t_4.\\
(d):\ \  g_1=t_3(t_1-t_2),&g_2=t_1(t_3-t_4),&g_3=t_2(t_1-t_4).
\end{eqnarray*}
Subcase (a): This case cannot occur because the ideal $(g_1,g_2,g_3)$ has height
$2$.  

Subcase (b): The reduced Gr\"obner basis of $I=(g_1,g_2,g_3)$ with
respect to the GRevLex order $\prec$ is given by 
\begin{eqnarray*}
g_1=t_1t_2-t_1t_3,&g_2=t_1t_4-t_3t_4,&g_3=t_2t_3-t_2t_4,\\
g_4=t_3^2t_4-t_2t_4^2,&
g_5=t_1t_3^2-t_2t_4^2,& 
g_6=t_2^2t_4^2-t_2t_4^3.
\end{eqnarray*}
Hence the binomial $h=t_2t_4-t_3t_4\notin I$ because $t_2t_4$ does not
belong to ${\rm in}_\prec(I)$, the initial ideal of $I$. Since
$h^2=-2t_4^2g_3+t_4g_4+g_6$, we get that $h\in{\rm rad}(I)$. Thus $I$
is not a radical ideal which is impossible because $I=I(\mathbb{X})$
is a vanishing ideal. Therefore this case cannot occur.

Subcase (c): In this case one has
$I=(t_1t_2-t_3t_4,t_1t_3-t_2t_4,t_2t_3-t_1t_4)$, as required. From
Lemma~\ref{nov12-14}, we obtain that $q$ is odd. 

Subcase (d): The reduced Gr\"obner basis of $I=(g_1,g_2,g_3)$ with
respect to the GRevLex order $\prec$ is given by 
\begin{eqnarray*}
h_1=t_2t_3-t_1t_4,&g_2=t_1t_3-t_1t_4,&g_3=t_1t_2-t_2t_4,\\
g_4=t_1t_4^2-t_2t_4^2,&
g_5=t_1^2t_4-t_2t_4^2,&
g_6=t_2^2t_4^2-t_2t_4^3.
\end{eqnarray*}
Setting $h=t_1t_4-t_2t_4$, as in Subcase (b), one can readily verify
that $h\notin I$ and $h^2\in I$. Hence $I$ is not a radical ideal.
Therefore this case cannot occur.

Case (ii): Assume that $s=3$. By hypothesis $I=I(\mathbb{X})$ is
generated by $2$ binomials $g_1,g_2$. By Lemma~\ref{nov10-14}(b) for each pair
$1\leq i<j\leq 3$ there are positive integers $c_{ij}$ and $a_{ij}$
such that $t_i^{c_{ij}}t_j$ and $t_it_j^{a_{ij}}$ occur as terms in
$g_1,g_2$. Since there are at most $4$ monomials that occur in 
the $g_i$'s it is seen that, up to permutation of variables, there
are $2$ subcases to consider: 
\begin{eqnarray*}
(a):\ & &g_1=t_1t_3-t_2t_3,\ g_2=t_1^{c_{12}}t_2-t_1t_2^{a_{12}}
\mbox{ with }c_{12}=a_{12}\geq 2.\\
(b):\  & & g_1=t_1t_2-t_2t_3,\ g_2=t_1t_3-t_2t_3.
\end{eqnarray*}
Subcase (a) cannot occur because the ideal
$I=(g_1,g_2)$, being contained in $(t_1-t_2)$, has height $1$. Thus we
are left with subcase (b), that is, $I=(t_1t_2-t_2t_3,t_1t_3-t_2t_3)$,
as required. 

Case (iii): If $s=2$, then $\mathbb{X}$ is parameterized by $y^{v_1}$,
$y^{v_2}$. Pick $y_k\in{\rm supp}(y^{v_1})\setminus{\rm
supp}(y^{v_2})$. Making $y_k=0$ and $y_\ell=1$ for $\ell\neq k$, we
get that $[e_2]\in\mathbb{X}$, and by a similar
argument $[e_1]\in\mathbb{X}$. As $\mathbb{X}\cup\{[0]\}$ is a monoid
under componentwise multiplication, by Lemma~\ref{nov14-14-1}, 
$I(\mathbb{X})$ has the required form.  

Case (iv): If $s=1$, this case is clear. 

$\Leftarrow$) The converse is clear because the vanishing ideal 
$I(\mathbb{X})$ has height $s-1$.
\end{proof}

\begin{proposition}\label{azucena-vila-ci-prop} 
If $I$ is an ideal of $S$ of one of the following forms:
\begin{itemize}
\item[(i)] $s=4$, $q$ is odd and 
$I=(t_1t_2-t_3t_4,t_1t_3-t_2t_4,t_2t_3-t_1t_4)$,
\item[(ii)] $s=3$ and $I=(t_1t_2-t_2t_3,t_1t_3-t_2t_3)$,
\item[(iii)] $s=2$ and $I=(t_1^{r+1}t_2-t_1t_2^{r+1})$, where
$0\neq r\in\mathbb{N}$ and $r$ divides $q-1$, 
\end{itemize}
then there is a set $\mathbb{X}$ in $\mathbb{P}^{s-1}$ of clutter type
 such that $I$ is the vanishing ideal $I(\mathbb{X})$. 
\end{proposition}

\begin{proof} The result follows from 
Lemma~\ref{nov12-14} and Remark~\ref{az12-macaulay2}, 
Lemma~\ref{nov12-14-1} and Remark~\ref{az13-mac2}, and 
Lemma~\ref{nov14-14}, respectively
\end{proof}

\begin{problem} Let $\mathbb{X}$ be a set of clutter type such that
$I(\mathbb{X})$ is a complete intersection. Using the techniques
of \cite{duursma-renteria-tapia,cartesian-codes,algcodes,ci-codes} and
Theorem~\ref{azucena-vila-ci} find formulas for 
the {\it basic parameters\/} of the Reed-Muller-type codes 
associated to $\mathbb{X}$.

\end{problem}

\bibliographystyle{plain}

\end{document}